\documentclass[11pt]{article}
\usepackage{epsfig}
\usepackage{amssymb,amsmath,amsthm,amscd}
\usepackage{latexsym}

\newtheorem{thm}{Theorem}[section]

\newtheorem{lem}[thm]{Lemma}

\theoremstyle{definition}
\newtheorem{defn}[thm]{Definition}
\newtheorem{rem}[thm]{Remark}

\newcounter{labelflag} 

\newcommand{\Label}[1]{
                       \ifnum\thelabelflag=1
                          \ifmmode
                             \makebox[0in][l]{\qquad\fbox{\rm#1}}
                          \else
                             \marginpar{\vspace{0.7\baselineskip}
                                        \hspace{-1.1\textwidth}
                                        \fbox{\rm#1}}
                          \fi
                       \fi
                       \label{#1} }

\newcommand{\be}{\begin{equation}}
\newcommand{\ee}{\end{equation}}

\newcommand{\R}{\mathbb{R}}
\newcommand{\N}{\mathbb{N}}
\newcommand{\E}{\mathbb{E}}

 \def  \eps { { \varepsilon } }

\begin{document}

\baselineskip=1.3\baselineskip

\begin{titlepage}
\title{\large  \bf \baselineskip=1.3\baselineskip
Non-autonomous  stochastic lattice systems with Markovian switching \footnote{This work was supported
 by NSFC (11971394), Central Government Funds for Guiding Local Scientific and Technological Development (2021ZYD0010)
  and Fundamental Research Funds for the Central Universities (2682021ZTPY057).}  }
\vspace{10mm}

\author{{ Dingshi Li, Yusen Lin\footnote{Corresponding authors:
linyusen@my.swjtu.edu.cn (Y. Lin).}, Zhe Pu}
 \\{ \small\textsl{
 School of Mathematics, Southwest Jiaotong University, Chengdu, 610031, P. R. China}}}

\date{}
\end{titlepage}

\maketitle

{ \bf Abstract} \ \ \
The aim of this paper is to study   the dynamical behavior of
 non-autonomous  stochastic  lattice systems with Markovian switching.
We first  show  existence  of  an evolution system of measures  of  the stochastic system.
We then study the pullback (or forward) asymptotic stability in distribution  of
the evolution system of measures.
 We finally prove that any limit point of a tight sequence of
an evolution system of measures of the   stochastic  lattice systems
must be  an evolution system of measures of the corresponding
 limiting  system as the intensity of noise converges zero.
 In particular, when the coefficients are periodic with respect to  time,
 we show every  limit point of a   sequence of
 periodic measures of the   stochastic   system
 must be a periodic measure of the limiting  system as  the  noise intensity
 goes to zero.



\medskip

\section{Introduction}
\setcounter{equation}{0}

Let $(W_k)_{k\in \mathbb N}$ be a sequence of independent
standard two-side Wiener processes on a complete filtered probability space $
( \Omega ,\mathcal F, \{\mathcal F_t  \}_{t\in \R} ,P )$  satisfying  the usual condition
and $r(t)$, $t\in \R$, be a right continuous Markov chain,  independent of
the Brownian motion $(W_k)_{k\in \mathbb N}$,  on the
probability space $
( \Omega ,\mathcal F, \{\mathcal F_t  \}_{t \in \R} ,P )$ taking values in a finite state space $S = \left\{
{1,2, \ldots ,N} \right\}$ with generator $\Gamma  = {\left(
{{r_{ij}}} \right)_{N \times N}}$ given by
\[P\left\{ {r\left( {t + \Delta } \right) = j\left| {r\left( t \right)
= i} \right.} \right\} = \left\{ {\begin{array}{*{20}{c}}
   {{r_{ij}}\Delta  + o\left( \Delta  \right),\;\;\;\ i \ne j;} \hfill  \\
   {1 + {r_{ij}}\Delta  + o\left( \Delta  \right),\;\;\;\ i = j,} \hfill  \\
\end{array}} \right.\]
where $\Delta  > 0$ and ${\lim _{\Delta  \to 0}}o\left( \Delta
\right)/\Delta  = 0, {r_{ij}} \ge 0$ is the transition rate from $i$
to $j$ if $i\ne j$ and ${r_{ii}} =  - \sum\nolimits_{i\ne j}
{{r_{ij}}} $. It
is well known that almost every sample path of $r(t)$ is a right-continuous step function
and $r(t)$ is ergodic.

In this paper, we study the limiting behavior of
 evolution system of measures of the nonautonomous stochastic  lattice system with Markovian switching defined
on the integer set $\mathbb Z$: for $s\in \R$,
 \begin{align}\label{eu1}
\begin{split}
 du_i \left( t \right)& - \nu \left( {u_{i - 1} \left( t \right) - 2u_i \left( t \right)
 + u_{i + 1} \left( t \right)} \right)dt + \lambda(r(t)) u_i \left( t \right)dt \\
  &= \left( {f_i \left( {t,r(t),u_i(t)} \right) + g_i(r(t)) } \right)dt\\
  &\quad + \varepsilon\sum\limits_{k = 1}^\infty  {\left( {h_{i,k}(r(t))
   + \sigma _{i,k}
   \left( {t,r(t),u_i \left( t \right)} \right)} \right)} dW_k \left( t \right),\quad t > s,
 \end{split}
\end{align}
with initial data
\begin{equation}\label{eu2}
u_i \left( s\right) = \xi _i  \quad \text{and}\quad r(s)=j\in S,
\end{equation}
where  $u=(u_i)_{i\in \mathbb Z}$ is an unknown sequence, $\xi=(\xi_i)_{i\in \mathbb Z}\in l^2$
is  given,  $0<\varepsilon \le 1$, $\nu>0$, for $j\in S$  $\lambda(j)>0$,  $g(j)=(g_i(j))_{i\in \mathbb Z}$
and $h(j)=(h_{i,k}(j))_{i\in \mathbb Z,k\in \mathbb N}$ are
given   in $l^2$, and $f_i, \sigma_{i,k}: \mathbb R\times S \times\mathbb R  \rightarrow \mathbb R$
are nonlinear   functions for  every  $i\in \mathbb Z$ and $k\in \mathbb N$.

 We mention that lattice systems have many applications
in practice and have been extensively investigated.
For stochastic lattice systems  without time-dependent forcing, the
existence of random attractors was proved in  \cite{H2019,CHSV2016,CMV2012,BLL2006,CL2008}. The existence
of random attractors was also obtained in  \cite{SWHP,WangX1,BLW2014} for the systems with time-dependent
forcing.
The dynamical behavior  of invariant measures
  of stochastic lattice systems was obtained   in \cite{CLW, WW1,WW2,W2019,LWW2021}
 and  the limiting behavior of periodic  measures of stochastic lattice systems with periodic forcing term
 was studied in \cite{LWW2022}. The concept of evolution system of measures was introduced by \cite{DG2006}.
 It is the natural generalization of the notion of an invariant measure to
non-autonomous systems.
 Recently, in \cite{WCT2022}, Wang et. al. studied the limiting behavior of evolution system of measures
 of non-autonomous stochastic lattice systems.   There is an extensive literature on
existence and stability of invariant measure for stochastic
 differential equations with Markovian switching, see e.g., \cite{YM2003,YZM2003,DDD2014}.
However, there is so far no result of evolution system of measures
of non-autonomous stochastic  lattice systems with Markovian switching.

This paper is concerned with the theory of evolution system of measures of
nonhomogeneous Markov processes generated by stochastic differential
equations with both non-autonomous deterministic and
Markovian switching. We will prove a sufficient  condition for existence
and pullback (or forward) asymptotic stability in distribution
 of evolution system of measures
 for such processes. We will also show the effect of evolution system of measures
 for a family of such processes from parameter disturbance.   For periodic Markov processes,
we prove the evolution system of measures are also periodic.
Periodic measures for SPDEs was  studied in \cite{DD2008,LL2021}.
As an application of our abstract results, we will investigate the existence,
  pullback (or forward) asymptotic stability in distribution,
 and the limiting behavior
  of  evolution system of measures of \eqref{eu1}-\eqref{eu2}
  as the noise intensity $\eps \to  0$.

The rest of this paper is organized as follows.
Section 2  is devoted to
 the existence,  stability and periodicity
of evolution system of measures of time
nonhomogeneous  Markov processes.
In Section 3, we show the limiting behavior
of evolution system of measures of time
nonhomogeneous  Markov processes.
Section 4 is devoted to the existence and uniqueness
of solutions to the stochastic  lattice system \eqref{eu1}-\eqref{eu2}.
In Section 5, we derive the uniform estimates
of solutions  which are needed
  for proving  our main results in later sections.
   In Section 6, we
   establish
the existence, stability and periodicity of evolution system of measures
on $l^2$ for \eqref{eu1}-\eqref{eu2}
 and prove   the convergence of evolution system of measures
of system \eqref{eu1}-\eqref{eu2} as $\varepsilon \to 0$.

\section{Existence and Stability}
\setcounter{equation}{0}

In what follows, we denote by $X$  a Polish space with a metric $d_X$  and
denote by $H$ a separable Banach space with norm  $\|\cdot\|_H$, respectively. Define
$C_b(X)$ as the space of bounded continuous functions $f:X\rightarrow \R$
endowed with the norm
\[
\left\| f \right\|_\infty   = \mathop {\sup }\limits_{x \in X} \left| {f\left( x \right)} \right|,
\]
and denote by $L_b(X)$ the space of bounded Lipschitz functions on $X$. That is,
of functions $f\in C_b(X)$ for which
\[
\text{Lip}\left( f \right): = \mathop {\sup }\limits_{x_1 ,x_2  \in X} \frac{{\left| {f\left( {x_1 } \right)}  -
{f\left( {x_2 } \right)} \right|}}{{\text{dist}_X \left( {x_1 ,x_2 } \right)}} < \infty .
\]
The space $L_b(X)$ is endowed with the norm
\[
\left\| f \right\|_L  = \left\| f \right\|_\infty   + \text{Lip}\left( f \right).
\]
Let us denote by $\mathcal P(X)$ the set of probability measures on $(X,\mathcal B(X))$.
Define a metric on $\mathcal P(X)$ by
\[
\text{d}_L^* \left( {\mu _1 ,\mu _2 } \right) = \mathop {\sup }\limits_{\scriptstyle f \in L_b \left( X \right) \hfill \atop
  \scriptstyle \left\| f \right\|_L  \le 1 \hfill} \left| {\left( {f,\mu _1 } \right) - \left( {f,\mu _2 } \right)} \right|,
  \quad \mu_1,\mu_2\in \mathcal P(X).
\]
Given $s,t\in \R$ and $s\leq t$, we let $r_{s,j}(t)$ be the Markov chain starting from state $j\in S$ at $t = s$ and
let $u(t,s,\xi,j)$ be a  stochastic process with initial conditions $u(s,s,\xi,j) = \xi \in H$ and
$r(s) = j$ at initial time $t=s$.
Let $y(t,s,\xi,j)$ denote the $(H \times S)$-valued process $(u (t,s,\xi,j), r_{s,j}(t))$ and  $y(t,s,\xi,j)$ be a time
nonhomogeneous  Markov process. Let $p(t,  s, \xi, j, (dy , {k}))$ denote the transition probability
of the process $y(t,s,\xi,j)$. For $A\subset \mathcal B(H)$ and $B\subset S$,
 let $P(t, s, \xi,  j, A \times B)$ denote the probability of event $\{y(t,s,\xi,j)\in A \times B\}$
given initial condition $y(s,s,\xi,j) = (\xi, j)$ at time $t=s$, i.e.,
\[
P\left( {t,s, \xi,j,A \times B} \right) = \sum\limits_{k \in B}
{\int_A {p\left( {t,s,\xi, j,(dy,k)} \right)} } .
\]

We define the transition evolution operator
\[
P_{s,t} \varphi \left( {\xi ,j} \right) = \E\left[ {\varphi \left( {u\left( {t,s,\xi ,j}
 \right),r_{s,j} \left( t \right)} \right)} \right], \quad \varphi\in C_b(H \times S).
\]
Assume that $P_{s,t}$ is Feller, that is, $P_{s,t}:C_b(H\times S)\rightarrow C_b(H\times S)$, for $s<t$.
Denote by $P_{s,t}^*: \mathcal P(H\times S)\rightarrow \mathcal P(H\times S)$ the duality
operator of $P_{s,t}$. For $(\xi,j)\in H\times S$,
denote by $\delta_{\xi ,j}$ the Dirac measure  concentrating   on  $(\xi,j)$.

In this section, we show existence and stability of   an evolution system of measures $
\left( {\mu _t } \right)_{t \in \R}$
 indexed by $\R$. An evolution system of measures $
\left( {\mu _t } \right)_{t \in \R}$ satisfies
 each $\mu _t$, $t\in \R$, is a probability measure on $H\times S$ and
 \[\sum\limits_{j \in S}
\int_{H} {P_{s,t} \varphi \left( \xi,j \right)\mu _s \left( {d\xi,j} \right)}  = \sum\limits_{j \in S}
\int_{H} {\varphi \left( \xi,j \right)\mu _t \left( {d\xi,j} \right),\quad\forall
 \varphi  \in C_b \left( H\times S \right)} ,\quad s < t.
\]
For $\varpi>0$, the evolution system of measures $\mu _t$, $t\in \R$ is $\varpi$-periodic, if
$$\mu _t=\mu _{t+\varpi},\quad  \forall t\in \R.
$$

We now recall the definition of pullback (or forward) asymptotic stability in distribution of the evolution system of measures.

\begin{defn}
The evolution system of measures $\left( {\mu _t } \right)_{t \in \R}$ of
Markov processes $y(t,s,\xi,j)$
is said to be
pullback  asymptotic stability in distribution if for any $\varphi\in C_b(H\times S)$,
\[
\mathop {\lim }\limits_{s \to  - \infty } P_{s,t} \varphi \left( \xi,j \right) =
\sum\limits_{j \in S}\int_{H} {\varphi \left(x,j \right)\mu _t \left( {dx,j} \right)} ,
\quad \forall t \in \R,\quad (\xi,j) \in H\times S,
\]
and be forward  asymptotic stability in distribution if for any $\varphi\in C_b(H\times S)$,
 \[
\mathop {\lim }\limits_{t \to +\infty } \left[ {P_{s,t} \varphi \left( \xi,j \right) -
\sum\limits_{j \in S}\int_{H} {\varphi \left( x,j \right)\mu _t \left( {dx,j} \right)} } \right]
= 0,\quad\forall s \in \R,\quad (\xi,j) \in H\times S.
\]
\end{defn}
Forward  asymptotic stability in distribution implies that $P_{s,t} \varphi \left( x \right)$
 approaches as $t\rightarrow +\infty$ a curve, parametrized
by $t$, which is independent of $s$ and $(\xi,j)$. This is the natural generalization of the strongly
mixing property for an autonomous dissipative system.

In the sequence, we always assume that

$( A_0)$  For
any $s\in \R$, $T>0$, bounded set $B\subset H$ and $\eta>0$, there exists a
constant $R=R(\eta,B,T)>0$, independent of $s$, such that
for any $\xi\in B$ and $j\in S$,
\[
P\left\{ {\|{u\left( {t,s,\xi ,j} \right)}\|_H \ge R},\, t\in [s,s+T] \right\} < \eta.
\]

\begin{defn}

The processes  $u(t,s,\xi,j)$  are said to have properties:

$(A_1)$  If for
any $s\in \R$,  $\xi\in H$ and $\eta>0$, there exists a  bounded
 subset $B=(\eta,\xi)$ of $H$, independent of $s$, such that
for any $\xi\in H$, $j\in S$ and $t\geq s$,
\[
P\left\{ {{u\left( {t,s,\xi ,j} \right)}\in B} \right\}>1-\eta .
\]

$(A_2)$  If for any $s\in \R$, $\eta>0$ and  bounded
 subset $B$ of $H$, there exists a $T=T(\eta,B)$, independent of $s$, such that
for $\left( {\xi _1 ,\xi _2 ,j} \right) \in B \times B \times S,$
\[
P\left\{ {\| {u\left( {t,s,\xi _1 ,j} \right)- u\left( {t,s,\xi _2 ,j} \right)} \|_H <
 \eta } \right\} \ge 1 - \eta,\quad \forall t-s \ge {\rm{T}}.
\]

\end{defn}

\begin{defn}

The processes  $u(t,s,\xi,j)$  are said to have properties:

$(A_1^*)$  If for
any $s\in \R$,  $\xi\in H$ and $\eta>0$, there exists a
compact  $K=(\eta, \xi)\subset H$, independent of $s$, such that
for any $\xi\in H$, $j\in S$ and $t\geq s$,
\[
P\left\{ {{u\left( {t,s,\xi ,j} \right)}\in K} \right\} >1-\eta.
\]

$(A_2^*)$  If for any $s\in \R$, $\eta>0$ and any compact
 subset $K$ of $H$, there exists a $T=T(\eta,K)$, independent of $s$, such that
for $\left( {\xi _1 ,\xi _2 ,j} \right) \in K \times K\times S,$
\[
P\left\{ {\| {u\left( {t,s,\xi _1 ,j} \right)- u\left( {t,s,\xi _2 ,j} \right)} \|_H <
 \eta } \right\} \ge 1 - \eta,\quad \forall t-s \ge {\rm{T}}.
\]

\end{defn}

\begin{lem}\label{Lcijp}
Assume that the processes  $u(t,s,\xi,j)$ have property $(A_2)$. Then, for any bounded
set $B\subset H$,
\[
\mathop {\lim }\limits_{s \to  - \infty }
d_{\rm{L}}^* \left( {P_{s,t}^ *  \delta_{\xi _1 ,i} ,P_{s,t}^ *  \delta_{\xi _2 ,j} } \right) = 0
\]
uniformly in $\xi_1,\xi_2\in B$ and $j_1,j_2 \in S$.

\end{lem}

\begin{proof}
Choose any nonotone decreasing sequence $\left\{ {s_n } \right\}_{n = 1}^\infty$
satisfying $s_n  < t,\,\,\forall n \in \N,$ and
 $s_n  \to  - \infty ,\,\, \text{as}\,\, n \to \infty.$
For any pair of $j_1,j_2\in S$ and $n\in \N$, define the stopping time
\[
\tau _{j_1,j_2}^n  = \inf \left\{ {t - s_n|t \ge s_n,r_{s_n,j_1} \left( t \right) = r_{s_n,j_2} \left( t \right)} \right\}.
\]
Recall that $r_{s_n,j}(t)$ is the Markov chain starting from state $j\in S$ at initial time $t = s_n$ and due to the
ergodicity, for any $n\in \N$,
 $\tau _{j_1,j_2}:=\tau _{j_1,j_2}^{n}<\infty$ a.s.
 By homogeny  of the Markov chain, for any $\eta>0$ and $n\in \N$,
there exists a positive number $T$, independent of $n$, such that
\begin{align*}
P\left\{ {\tau _{j_1,j_2}  \le T} \right\} > 1 - \frac{\eta }{8},\quad \forall {j_1,j_2} \in S.
\end{align*}
For such $T$, by $(A_0)$, there is a sufficiently large $R>0$ such that
\begin{align}\label{cij1}
P\left( {\Omega _{\xi ,j} } \right) > 1 - \frac{\eta }{{16}},\quad \forall \left( {\xi ,j} \right) \in B \times S,
\end{align}
where $\Omega _{\xi ,j}  = \left\{ {\left\| {u\left( {t,s_n,\xi ,j} \right)}
 \right\|_H \le R,\forall t \in \left[ {s_n,s_n + T} \right]} \right\}.$
For any given $\xi_1,\xi_2\in B$ and $j_1,j_2 \in S$, set
$\Omega ^{'}  = \Omega _{\xi _1 ,j_1}  \cap \Omega _{\xi _2 ,j_2}.$
For any $\varphi\in L_b(H\times S)$ and $s_n<t-T$,
\begin{align}\label{cij2}
\begin{split}
 &\left| {\left( {\varphi ,P_{s_n,t}^ *  \delta_{\xi _1 ,j_1} } \right) -
 \left( {\varphi ,P_{s_n,t}^ *  \delta_{\xi _2 ,j_2} } \right)} \right| \\
  &= \left| {\E\varphi \left( {u\left( {t,s_n,\xi _1 ,j_1} \right),r_{\left\{ {s_n,j_1} \right\}} \left( t \right)} \right) -
   \E\varphi \left( {u\left( {t,s_n,\xi _2 ,j_2} \right),r_{\left\{ {s_n,j_2} \right\}} \left( t \right)} \right)} \right| \\
  &\le 2P\left\{ {\tau _{j_1,j_2}  > T} \right\} + \E\left( {I_{\left\{ {\tau _{j_1,j_2}  \le T} \right\}}
  \left| {\varphi \left( {u\left( {t,s_n,\xi _1 ,j_1} \right),r_{\left\{ {s_n,j_1} \right\}} \left( t \right)} \right)
   - \varphi \left( {u\left( {t,s_n,\xi _2 ,j_2} \right),r_{\left\{ {s_n,j_2} \right\}} \left( t \right)} \right)} \right|} \right) \\
  &\le \frac{\eta}{4} + \E\left( {I_{\left\{ {\tau _{j_1,j_2}  \le T} \right\}}
   \E\left( {\left| {\varphi \left( {u\left( {t,s_n,\xi _1 ,j_1} \right),r_{\left\{ {s_n,i} \right\}}
    \left( t \right)} \right) - \varphi \left( {u\left( {t,s_n,\xi _2 ,j_2} \right),r_{\left\{ {s_n,j_2} \right\}}
    \left( t \right)} \right)} \right||\mathcal F_{\tau _{j_1,j_2} } } \right)} \right) \\
 & \le \frac{\eta }{4} + \E\left( {I_{\left\{ {\tau _{j_1,j_2}  \le T} \right\}}
 \E\left( {\left| {\varphi \left( {u\left( {t ,\tau _{j_1,j_2} ,\xi _1^* ,k}
 \right),r_{\left\{ {\tau _{j_1,j_2} ,k} \right\}} \left( t \right)} \right) - \varphi
  \left( {u\left( {t  ,\tau _{j_1,j_2} ,\xi _2^* ,k} \right),r_{\left\{ {\tau _{j_1,j_2} ,k}
   \right\}} \left( t \right)} \right)} \right|} \right)} \right) \\
 & \le \frac{\eta }{4} + \E\left( {I_{\left\{ {\tau _{j_1,j_2}  \le T} \right\}}
 \E\left( {2 \wedge \| {u\left( {t,\tau _{j_1,j_2} ,\xi _1^* ,k} \right)-
  u\left( {t ,\tau _{j_1,j_2} ,\xi _2^* ,k} \right)} \|_H} \right)} \right) \\
& \le \frac{\eta }{4} + 2P\left( {\Omega  - \Omega ^{'} } \right) +
 \E\left( {I_{\Omega ^{'} \cap \left\{ {\tau _{j_1,j_2}  \le T} \right\}}
  \E\left( {2 \wedge \| {u\left( {t ,\tau _{j_1,j_2} ,\xi _1^* ,k} \right)-
    u\left( {t  ,\tau _{j_1,j_2} ,\xi _2^* ,k} \right)} \|_H} \right)} \right),
 \end{split}
 \end{align}
where $\xi _1^ *   = u\left( {\tau _{j_1,j_2} ,s_n ,\xi _1 ,j_1} \right)
$, $\xi _2^ *   = u\left( {\tau _{j_1,j_2} ,s_n ,\xi _2 ,j_2} \right)$
and $k = r_{s_n ,j_1} \left( {\tau _{j_1,j_2} } \right) = r_{s_n ,j} \left( {\tau _{j_1,j_2} } \right).$
Note that given $\omega  \in \Omega ^{'}  \cap \left\{ {\tau _{j_1,j_2}  \le T} \right\},
$ $\left\| {\xi _1^* } \right\| \vee \left\| {\xi _2^* } \right\| \le R$.
So, by $(A_2)$,
there exists a constant $T_1>T$  such that
\begin{align}\label{cij3}
\E\left( {2 \wedge \|{u\left( {t  ,\tau _{j_1,j_2} ,\xi _1^* ,k} \right)-
 u\left( {t,\tau _{j_1,j_2} ,\xi _2^* ,k} \right)} \|_H} \right) < \frac{\eta }{2},\quad t-s_n>T_1 .
\end{align}
It therefore follows from \eqref{cij1}-\eqref{cij3} that
\[
\left| {\E\varphi \left( {u\left( {t,s_n,\xi _1 ,j_1} \right),r_{\left\{ {s_n,j_1} \right\}}
 \left( t \right)} \right) - \E\varphi \left( {u\left( {t,s_n,\xi _2 ,j_2} \right),r_{\left\{ {s_n,j_1} \right\}}
  \left( t \right)} \right)} \right| \le \frac{\eta}{4}
  + \frac{\eta }{4} + \frac{\eta }{2}, \quad t-s_n>T_1 .
\]
Since $\varphi$ is arbitrary, we must have
\[
d_{\rm{L}}^* \left( {P_{s_n,t}^ *  \delta_{\xi _1 ,j_1} ,P_{s_n,t}^ *
 \delta_{\xi _2 ,j_2} } \right)\leq \eta, \quad t-s_n>T_1 ,
\]
for all $\xi _1,\xi _2\in B$ and $j_1,j_2\in S$.
By the $\left\{ {s_n } \right\}_{n = 1}^\infty$ chosen arbitrarily, the proof is completed.

\end{proof}

Repeating the scheme used in the proof Lemma \ref{Lcijp}, we get the following result.

\begin{lem}\label{Lcijf}
Assume that the processes  $u(t,s,\xi,j)$  have property $(A_2)$. Then, for any bounded
subset $B$ of $H$,
\[
\mathop {\lim }\limits_{t \to  + \infty }
d_{\rm{L}}^* \left( {P_{s,t}^ *  \delta_{\xi _1 ,j_1} ,P_{s,t}^ *  \delta_{\xi _2 ,j_2} } \right) = 0
\]
uniformly in $\xi_1,\xi_2\in B$ and $j_1,j_2 \in S.$

\end{lem}

\begin{lem}\label{Lcau}
Assume that the processes  $u(t,s,\xi,j)$ have properties $(A_1)$ and $(A_2)$. Then  for any  $t\in \R$ and $(\xi,j)\in H \times S$,
there exists a  $\mu_t\in \mathcal P(H\times S)$, independent of $(\xi,j)$, such that
\[
\mathop {\lim }\limits_{s \to  - \infty } d_L^* \left( {P_{s,t}^ *  \delta_{\xi,j} ,\mu _t } \right) = 0.
\]
\end{lem}
\begin{proof}
Fix any $t\in \R$ and $(\xi, j)\in H\times S$. We first claim that
$\{P_{s,t}^ *  \delta_{\xi  ,j}:s\leq t\}$  is Cauchy in the space $\mathcal P(H \times S)$ with
metric $d_L^*$. To end this,
 we need to show that for any $\eta$, there is a $T>0$
such that
\begin{align}\label{cau0}
d_L^* \left( {P_{s - h,t}^ *  \delta_{\xi ,j} ,P_{s,t}^ *  \delta_{\xi ,j} } \right) <
\eta ,\quad \forall s < t - T,\quad h > 0.
 \end{align}
This is equivalent to
\[
\left| {\E\varphi \left( {u\left( {t,s - h,\xi,j} \right),r_{\left\{ {s-h,j} \right\}} \left( t \right)} \right)
- \E\varphi \left( {u\left( {t,s,\xi,j} \right),r_{\left\{ {s,j} \right\}} \left( t \right)} \right)} \right| < \eta,
\]
where $\varphi\in L_b(H\times S)$.
For any $\varphi\in L_b(H\times S)$ and $h>0$, compute
\begin{align}\label{cau1}
\begin{split}
& \left| {\E\varphi \left( {u\left( {t,s - h,\xi ,j} \right),r_{\left\{ {s - h,j} \right\}} \left( t \right)} \right) -
 \E\varphi \left( {u\left( {t,s,\xi ,j} \right),r_{\left\{ {s,i} \right\}} \left( t \right)} \right)} \right| \\
  &= \left| {\E\left( {\E\left( {\varphi \left( {u\left( {t,s - h,\xi ,j} \right),
  r_{\left\{ {s - h,j} \right\}} \left( t \right)} \right)|\mathcal F_s } \right)} \right) -
   \E\varphi \left( {u\left( {t,s,\xi ,j} \right),r_{\left\{ {s,j} \right\}} \left( t \right)} \right)} \right| \\
  &= \left| \sum\limits_{j \in S}{\int_{H} {\E\varphi \left( {u\left( {t,s,z,k} \right),r_{\left\{ {s,k} \right\}} \left( t \right)} \right)
  p\left( {s,s - h,\xi ,j,(dz, \{ k\}) } \right) - \E\varphi \left( {u\left( {t,s,\xi ,j} \right),r_{\left\{ {s,j} \right\}}
  \left( t \right)} \right)} } \right| \\
  &\le \sum\limits_{j \in S}\int_{H} {\left| {\E\varphi \left( {u\left( {t,s,z,k} \right),r_{\left\{ {s,k} \right\}} \left( t \right)} \right) -
  \E\varphi \left( {u\left( {t,s,\xi ,j} \right),r_{\left\{ {s,j} \right\}} \left( t \right)} \right)}
   \right|p\left( {s,s - h,\xi ,j,(dz,\{ k\}) } \right)}  \\
  &\le 2P\left( {s,s-h,\xi,j, B _R^C  \times S} \right) \\
  &\quad+ \sum\limits_{j \in S}\int_{B_R} {\left| {\E\varphi \left( {u\left( {t,s,z,k} \right),
  r_{\left\{ {s,k} \right\}} \left( t \right)} \right) -
   \E\varphi \left( {u\left( {t,s,\xi ,j} \right),r_{\left\{ {s,j} \right\}} \left( t \right)} \right)}
    \right|p\left( {s,s - h,\xi ,j,(dz, \{ k\}) } \right)},
 \end{split}
 \end{align}
where $B_R  = \left\{ {x \in H|\left\| x \right\|_H  \le R} \right\}$
and $B_R^C  =H-B_R$.
By $(A_1)$, there is
a positive number $R$ sufficiently large for
\begin{align}\label{cau2}
P\left( {s,s-h,\xi ,j, B _R^C  \times S} \right) < \frac{\eta }{4},\quad s<t.
 \end{align}
On the other hand, by Lemma \ref{Lcijp}, there is a $T>0$ such that
\begin{align}\label{cau3}
\mathop {\sup }\limits_{\varphi \in L_b(H\times S)} \left| {\E\varphi \left( {u\left( {t,s,z,k} \right),
r_{\left\{ {s,k} \right\}} \left( t \right)} \right)
- \E\varphi \left( {u\left( {t,s,\xi ,j} \right),r_{\left\{ {s,j} \right\}} \left( t \right)} \right)} \right|
< \frac{\eta }{2},\quad s < t - T,
 \end{align}
whenever $\left( {z,l} \right) \in B_R  \times S.$ Substituting \eqref{cau2} and \eqref{cau3} into \eqref{cau1} yields
\[
\left| {\E\varphi \left( {u\left( {t,s - h,\xi ,j} \right),r_{\left\{ {s - h,j} \right\}} \left( t \right)} \right)
- \E\varphi \left( {u\left( {t,s,\xi ,j} \right),r_{\left\{ {s,j} \right\}}
\left( t \right)} \right)} \right| < \eta,\quad \forall s < t - T,\quad h > 0.
\]
Since $\varphi$ is arbitrary, the desired inequality \eqref{cau0} must hold, i.e.,
$\{P_{s,t}^ *  \delta_{\xi  ,j}:s\leq t\}$  is Cauchy in the space $\mathcal P(H \times S)$ with
metric $d_L^*$. So, for any $t\in \R$ and $(\xi,j)\in H\times S$ there
is a unique $\mu_t(\xi,j)\in \mathcal P(H \times S)$ such that
\[
\mathop {\lim }\limits_{s \to  - \infty } d_L^* \left( {P_{s,t}^ *  \delta_{\xi,j} ,\mu _t(\xi,j) } \right) = 0.
\]
It remains to show that $\mu_t(\xi,j)$ is independent of $(\xi,j)$.
Now, for any $(\xi,j)\in H \times S$, by Lemma \ref{Lcijp},
\begin{align}\label{pcij}
\begin{split}
 &\mathop {\lim }\limits_{s \to  - \infty } d_L^* \left( {P_{s,t}^ *  \delta_{\xi ,j} ,\mu _t(0,1) } \right) \\
 &\quad \le \mathop {\lim }\limits_{s \to  - \infty } d_L^* \left( {P_{s,t}^ *  \delta_{\xi ,j} ,P_{s,t}^ *  \delta_{0,1} } \right)
  + \mathop {\lim }\limits_{s \to  - \infty } d_L^* \left( {P_{s,t}^ *  \delta_{0,1} ,\mu _t(0,1) } \right) = 0,
 \end{split}
 \end{align}
which implies that  $\mu_t$ is independent of $(\xi,j)$. This completes the proof of the lemma.
\end{proof}

\begin{thm}\label{Lev}
Assume that the processes  $u(t,s,\xi,j)$  have properties $(A_1)$ and $(A_2)$. Then, the family measures
 $\{\mu_t\}_{t\in \R}$ obtained above is an evolution system of measures on $H\times S$, i.e.,
\[
\sum\limits_{j \in S}\int_H {P_{s,t} \varphi \left( \xi,j \right)\mu _s \left( {d\xi,j} \right)}  =
\sum\limits_{j \in S}\int_H {\varphi \left( x,j \right)\mu _t \left( {d\xi,j} \right),\quad\forall
 \varphi  \in C_b \left( H\times S \right)} ,\quad s \leq t.
\]
\end{thm}
\begin{proof}
Let $s<\tau<t$ and $(\xi,j)\in H\times S$. We have from Lemma \ref{Lcau} for any $\varphi\in C_b \left( H\times S \right)$
\begin{align}\label{ev1}
\mathop {\lim }\limits_{s \to  - \infty } P_{s,\tau } \varphi \left( {\xi ,j} \right) = \left( {\varphi ,\mu _\tau  } \right).
\end{align}
Letting $s\rightarrow -\infty$ in the identity
\[
P_{s,\tau } P_{\tau ,t} \varphi \left( {\xi ,j} \right) = P_{s,t} \varphi \left( {\xi ,j} \right),
\]
 recalling Feller property and taking account \eqref{ev1} yields
\[
\left( {P_{\tau ,t} \varphi ,\mu_\tau  } \right) = \left( {\varphi ,\mu_t } \right).
\]
This completes the proof.
\end{proof}

\begin{rem} \label{rpc}
Lemma \ref{Lcau} and \ref{Lev} means that  the evolution
 system of measures $\{\mu_t\}_{t\in \R}$ obtained above is
pullback asymptotic stability in distribution.
\end{rem}

The following result gives information on the asymptotic behaviour of $P_{s,t}^ *  \delta_{\xi ,j}$ when
$t\rightarrow +\infty$.

\begin{thm}\label{Lfc}
Assume that the processes  $u(t,s,\xi,j)$  have properties $(A_1)$ and $(A_2)$ and the evolution system of  measures
 $\{\mu_t\}_{t\in \R}$ is obtained above.
For any $s \in \R$ and $(\xi,j)\in H\times S$, we have
\[
\mathop {\lim }\limits_{t \to  + \infty } d_{\rm{L}}^ * \left( {P_{s,t}^ *  \delta_{\xi ,j} ,\mu_t } \right) = 0.
\]
That is, the evolution system of measures $\{\mu_t\}_{t\in \R}$ is
forward  asymptotic stability in distribution.
\end{thm}

\begin{proof}
Fix any $s \in \R$ and $(\xi, j)\in H \times S$. To end the proof,
we need to show that for any $\eta>0$, there is a $T=T(\eta)>0$
such that
\[
d_L^ * \left( {P_{s,t}^ *  \delta_{\xi ,j} ,\mu_t } \right) < \eta,\quad \forall t>s+T.
\]
Notice that for $s_1<s<t$
\[
d_L^ * \left( {P_{s,t}^ *  \delta_{\xi ,j} ,\mu_t } \right)\leq
d_L^ * \left( {P_{s,t}^ *  \delta_{\xi ,j} ,{P_{s_1,t}^ *  \delta_{\xi ,j} }} \right)
+d_L^ * \left( {P_{s_1,t}^ *  \delta_{\xi ,j} ,\mu_t } \right).
\]
It follows from Lemma \ref{Lcau} that there is a $s^*=s^*(\eta)<s$
such that
\[
d_L^ * \left( {P_{s_1,t}^ *  \delta_{\xi ,j} ,\mu_t } \right) < \frac{\eta }{2},\quad \forall s_1\leq s^*.
\]
It remains to show that there is a $T=T(\eta)>0$
such that
\begin{align}\label{fc}
d_L^ * \left( {P_{s,t}^ *  \delta_{\xi ,j} ,P_{s^ *,t}^ *  \delta_{\xi ,j}} \right) < \frac{\eta}{2} ,\quad \forall t>s+T.
 \end{align}
This is equivalent to
\[
\left| {\E\varphi \left( {u\left( {t,s^ *,\xi,j} \right),r_{\left\{ {s^ *,j} \right\}} \left( t \right)} \right)
- \E\varphi \left( {u\left( {t,s,\xi,j} \right),r_{\left\{ {s,j} \right\}}
\left( t \right)} \right)} \right| < \frac{\eta}{2},
\]
for any $\varphi\in L_b(H\times S)$ and $  t>s+T$. Compute
\begin{align}\label{fc1}
\begin{split}
& \left| {\E\varphi \left( {u\left( {t,s^ *,\xi ,j} \right),r_{\left\{ {s^ *,j} \right\}} \left( t \right)} \right) -
 \E\varphi \left( {u\left( {t,s,\xi ,j} \right),r_{\left\{ {s,j} \right\}} \left( t \right)} \right)} \right| \\
  &= \left| {\E\left( {\E\left( {\varphi \left( {u\left( {t,s^ *,\xi ,j} \right),
  r_{\left\{ {s^ *,i} \right\}} \left( t \right)} \right)|\mathcal F_s } \right)} \right) -
   \E\varphi \left( {u\left( {t,s,\xi ,j} \right),r_{\left\{ {s,j} \right\}} \left( t \right)} \right)} \right| \\
  &= \left| {\sum\limits_{j \in S}\int_{H} {\E\varphi \left( {u\left( {t,s,z,k} \right),r_{\left\{ {s,k} \right\}} \left( t \right)} \right)
  p\left( {s,s^ *,\xi ,j,(dz , k) } \right) - \E\varphi \left( {u\left( {t,s,\xi ,j} \right),r_{\left\{ {s,j} \right\}}
  \left( t \right)} \right)} } \right| \\
  &\le \sum\limits_{j \in S}\int_{H} {\left| {\E\varphi \left( {u\left( {t,s,z,k} \right),r_{\left\{ {s,k} \right\}} \left( t \right)} \right) -
  \E\varphi \left( {u\left( {t,s,\xi ,j} \right),r_{\left\{ {s,j} \right\}} \left( t \right)} \right)}
   \right|p\left( {s,s^ *,\xi ,j,(dz , k) } \right)}  \\
  &\le 2P\left( {s,s^ *,x,j,B _R^C  \times S} \right) \\
  &\quad+\sum\limits_{j \in S} \int_{B_R } {\left| {\E\varphi \left( {u\left( {t,s,z,k} \right),
  r_{\left\{ {s,k} \right\}} \left( t \right)} \right) -
   \E\varphi \left( {u\left( {t,s,\xi ,j} \right),r_{\left\{ {s,j} \right\}} \left( t \right)} \right)}
    \right|p\left( {s,s^ *,\xi ,j,(dz , k) } \right)}.
 \end{split}
 \end{align}
By $(A_1)$, there is
a positive number $R$ sufficiently large for
\begin{align}\label{fc2}
P\left( {s,s^*,\xi ,j, B _R^C  \times S} \right) < \frac{\eta }{4}.
 \end{align}
On the other hand, by Lemma \ref{Lcijf}, there is a $T=T(\eta)>0$ such that
\begin{align}\label{fc3}
\mathop {\sup }\limits_{\varphi \in L_b(H\times S)} \left| {\E\varphi \left( {u\left( {t,s,z,k} \right),
r_{\left\{ {s,k} \right\}} \left( t \right)} \right)
- \E\varphi \left( {u\left( {t,s,\xi ,j} \right),r_{\left\{ {s,j} \right\}} \left( t \right)} \right)} \right|
< \frac{\eta}{4},\quad t>s+T,
 \end{align}
whenever $\left( {z,l} \right) \in B_R  \times S.$ Substituting \eqref{fc3} and \eqref{fc2} into \eqref{fc1} yields
\[
\left| {\E\varphi \left( {u\left( {t,s^*,\xi ,j} \right),r_{\left\{ {s - h,i} \right\}} \left( t \right)} \right)
- \E\varphi \left( {u\left( {t,s,\xi ,j} \right),r_{\left\{ {s,j} \right\}}
\left( t \right)} \right)} \right| < \frac{\eta}{2} ,\quad \forall t>s+T.
\]
Since $\varphi$ is arbitrary, the desired inequality \eqref{fc} must hold.
The proof is completed.
\end{proof}

\begin{thm}\label{Tper}
Assume that the processes  $u(t,s,\xi,j)$  have properties $(A_1)$ and $(A_2)$ and the  $y(t,s,\xi,j)$
are $\varpi$-periodic Markov processes. Then  for any   $(\xi,j)\in H \times S$,
there exists a unique $\varpi$-periodic evolution system of
measures $\{\mu_t\}_{t\in \R}$, independent of $(\xi,j)$, such that
\[
\mathop {\lim }\limits_{s \to  - \infty } d_L^* \left( {P_{s,t}^ *  \delta_{\xi,j} ,\mu _t } \right) = 0,\quad \forall t\in \R,
\]
and
\[
\mathop {\lim }\limits_{t \to   +\infty } d_L^* \left( {P_{s,t}^ *  \delta_{\xi,j} ,\mu _t } \right) = 0,\quad \forall s\in \R.
\]
\end{thm}
\begin{proof}
It follows from Lemma \ref{Lev}, Remark \ref{rpc} and Lemma \ref{Lfc} that for any   $(\xi,j)\in H \times S$,
there exists an evolution system of measures  $\{\mu_t\}_{t\in \R}\subset \mathcal P(H\times S)$, independent of $(\xi,j)$, such that
\[
\mathop {\lim }\limits_{s \to  - \infty } d_L^* \left( {P_{s,t}^ *  \delta_{\xi,j} ,\mu _t } \right) = 0
\]
and
\[
\mathop {\lim }\limits_{t \to   +\infty } d_L^* \left( {P_{s,t}^ *  \delta_{\xi,j} ,\mu _t } \right) = 0.
\]
In the following we shall show that
$\mu_t$ is periodic. Take  subsequence $\{P_{s_n,t}^ *  \delta_{\xi  ,j}:s_n=t-n\varpi,\,n\in \N\}$  of $\{P_{s,t}^ *  \delta_{\xi  ,j}:s\leq t\}$
and   subsequence $\{P_{s_n,t+\varpi}^ *  \delta_{\xi  ,j}:s_n=t-(n-1)\varpi,\,n\in \N\}$
of $\{P_{s,t+\varpi}^ *  \delta_{\xi  ,j}:s\leq t\}$, respectively.
Since the processes $y(t,s,\xi,j)$
are $\varpi$-periodic, $P_{s_n,t+\varpi}^ *  \delta_{\xi  ,j}=P_{s_n-\varpi,t}^ *  \delta_{\xi  ,j}$.
This means the sequences  $\{P_{s_n,t}^ *  \delta_{\xi  ,j}:s_n=t-n\varpi,\,n\in \N\}$
and $\{P_{s_n,t+\varpi}^ *  \delta_{\xi  ,j}:s_n=t-(n-1)\varpi,\,n\in \N\}$
are same. Consequently, for any $t\in \R$, $\mu_t=
\mathop {\lim }\limits_{{\rm{s}} \to  - \infty }P_{s,t}^ *  \delta_{\xi  ,j}
=\mathop {\lim }\limits_{{\rm{s}} \to  - \infty }P_{s,t+\varpi}^ *  \delta_{\xi  ,j}=\mu_{t+\varpi}$.
It remains to prove uniqueness.
Assume $\{\mu_t\}_{t\in \R}$ and $\{\nu_t\}_{t\in \R}$ are the evolution systems of measures
of $y(t,s,\xi,j)$. By Fubini's theorem, we have for any $\varphi\in C_b(H\times S)$ and $s<t$
\begin{align}\label{mu1}
\begin{split}
&  \left| {\left( {\varphi , \mu _t } \right) - \left( {\varphi , \nu _t } \right)} \right|=
\left| {\left( {\varphi ,P_{s,t}^ *  \mu _s } \right) - \left( {\varphi ,P_{s,t}^ *  \nu _s } \right)} \right|
= \left| {\left( {P_{s,t} \varphi ,\mu _s } \right) - \left( {P_{s,t} \varphi ,\nu _s } \right)} \right| \\
 & = \left| {\sum\limits_{j_1  \in S} {\int_H {\E\left( {\varphi \left( {t,s,\xi _1 ,j_1 } \right)} \right)
 \mu _s \left( {d\xi _1 ,j_1 } \right) - \sum\limits_{j_2  \in S} {\int_H {\E\left( {\varphi \left( {t,s,\xi _2 ,j_2 }
  \right)} \right)\nu _s \left( {d\xi _2 ,j_2 } \right)} } } } } \right| \\
&  \le \sum\limits_{j_1 ,j_2  \in S} {\int_H {\int_H {\left| {\E\left( {\varphi \left( {t,s,\xi _1 ,j_1 } \right)
- \varphi \left( {t,s,\xi _2 ,j_2 } \right)} \right)} \right|\mu _s
\left( {d\xi _1 ,j_1 } \right)\nu _s \left( {d\xi _2 ,j_2 } \right)} } },
 \end{split}
 \end{align}
which together with Lemma \ref{Lcijp} and the Lebesgue dominated convergence theorem
implies that
\begin{align*}\label{mu2}
\begin{split}
 &\left| {\left( {\varphi ,\mu _t } \right) - \left( {\varphi ,\nu _t } \right)} \right| \le  \\
& \mathop {\lim }\limits_{n \to   \infty } \sum\limits_{j_1 ,j_2  \in S}
 {\int_H {\int_H {\left| {\E\left( {\varphi \left( {t,-n \varpi,\xi _1 ,j_1 } \right) -
 \varphi \left( {t,-n \varpi,\xi _2 ,j_2 } \right)} \right)} \right|\mu _{-n \varpi} \left( {d\xi _1 ,j_1 }
 \right)\nu _{-n \varpi} \left( {d\xi _2 ,j_2 } \right)} } }  = 0.
 \end{split}
 \end{align*}
Then we have for any $\varphi\in C_b(H\times S)$
\[
\left( {\varphi ,\mu _t } \right) = \left( {\varphi ,\nu _t } \right).
\]
The proof is completed.
\end{proof}

\begin{rem}
When we replace $(A_1)$-$(A_2)$  with Condition $(A_1^*)$-$(A_2^*)$,
by minor modifying to the process of proof,
the   all results above are correct.

\end{rem}

\section{Limits    of  evolution system of measures}
\setcounter{equation}{0}

Suppose for every $\varepsilon\in[0,1]$, $s,t\in \R$ and $s\leq t$, $u^\varepsilon(t,s,\xi,j)$ be a  stochastic
process with initial conditions $u^\varepsilon(s,s,\xi,j) = \xi \in H$ and
$r(s) = j\in S$ at initial time $t=s$.
Let $y^\varepsilon(t,s,\xi,j)$ denote the $(H \times S)$-valued
process $(u^\varepsilon(t,s,\xi,j), r_{s,j}(t))$. $y^\varepsilon(t,s,\xi,j)$ are time
nonhomogeneous  Markov process and its probability transition operators
are Feller.

We assume that

$(A_3)$ For every
  compact set $K\subset H$,  $\varepsilon_0 \in [0,1]$
   and $\eta>0$,
\begin{equation}\label{haa}
\mathop {\lim }\limits_{\varepsilon  \to  \varepsilon_0}
 \mathop {\sup }\limits_{(x,j) \in K\times S}
P\left (
{\| {y^\varepsilon  \left( {t, s, \xi,j} \right)-
y^ {\varepsilon_0} \left( {t,s,\xi,j} \right)} \|_H \ge \eta }
\right ) = 0.
\end{equation}

\begin{thm}\label{Tmc}
Assume $(A_3)$ holds
and $\varepsilon_n \to \varepsilon_0\in [0,1]$.
Let $\{\mu_t\}_{t\in \R}$ be a family of   probability
measures  on $H\times S$  and let
 $\{\mu_t^{\varepsilon_n}\}_{t\in \R}$ be an evolution system of measures
of  $y^{\varepsilon_n}(t,s,\xi,j)$. If   for any $t\in \R$
  $\mu^{\varepsilon_n}_t\rightarrow \mu_t$ weakly, as $n\rightarrow \infty$,
then $\{\mu_t\}_{t\in \R}$ must be  an evolution system of measures    of $y^{ \varepsilon_0}(t,s,\xi,j)$.
\end{thm}

\begin{proof}
We only need to verify
  that for every    $\varphi\in L_b(H\times S)$ and $s<t$,
\be\label{haa_a1}
\sum\limits_{j \in S}\int_{H} {\E \varphi\left( {u^{\varepsilon_0}  \left( {t,s,\xi,j} \right),r_{s,j}(t)} \right)} \mu_s
\left( {d\xi,j} \right) = \sum\limits_{j \in S}\int_{H} {\varphi\left( \xi,j\right)\mu_t \left( {d\xi,j} \right)}.
\ee
Since for $t\in \R$, $\{\mu^{\varepsilon_n}_t\}$
 is tight, we see that for every $\epsilon>0$, there exists a
compact set $K=K(\epsilon,t) \subset H$ such that
\be\label{hbb}
 \mu ^{\varepsilon _n }_t
\left( K\times S \right) \ge 1 - \epsilon \quad \text{for all } \ n \in \N.
\ee
By \eqref{hbb}  we obtain
\begin{align}\label{hcc}
\begin{split}
& \left| {\sum\limits_{j \in S}\int_{H} {\E \varphi\left( {u^ {\varepsilon_0} \left( {t,s,\xi,j} \right),r_i(t)}
\right)} \mu^{\varepsilon _n}_s \left( {dx,i} \right) -
 \sum\limits_{j \in S}\int _{H}{\varphi\left( \xi,j \right)\mu ^{\varepsilon _n }_t \left( {d\xi,j} \right)} } \right| \\
  &= \left| {\sum\limits_{j \in S}\int_{H} {\E \varphi\left( {u^ {\varepsilon_0} \left( {t,s,\xi,j} \right),r_{s,j}(t)}
   \right)} \mu ^{\varepsilon _n }_s
  \left( {d\xi,j} \right) - \sum\limits_{j \in S}\int_{H} {\E \varphi\left( {u^{\varepsilon _n }
  \left( {t,s,\xi,j} \right),r_{s,j}(t)}
   \right)\mu ^{\varepsilon _n }_s \left( {d\xi,j} \right)} } \right| \\
 & \le \sum\limits_{j \in S}\int_{H} {\E\left| {\varphi\left( {u^ {\varepsilon_0} \left( {t,s,\xi,j} \right),r_{s,j}(t)} \right)
 - \varphi\left( {u^{\varepsilon _n } \left( {t,s,\xi,j} \right),r_i(t)} \right)} \right|}
 \mu ^{\varepsilon _n }_s \left( {d\xi,j} \right)\\
 & \le \sum\limits_{j \in S}\int_{K} {\E\left| { \varphi\left( {u^ {\varepsilon_0} \left( {t,s,\xi,j} \right),r_{s,j}(t)} \right)
 - \varphi\left( {u^{\varepsilon _n } \left( {t,s,\xi,j} \right),r_{s,j}(t)} \right)} \right|}
  \mu ^{\varepsilon _n }_s \left( {d\xi,j} \right)\\
  &\quad + 2\epsilon \sup_{(\xi,j)\in H\times S} |\varphi(\xi,j)|.
  \end{split}
 \end{align}
 Since $\varphi\in L_b(H\times S)$, for every  $\epsilon>0$, there exists   $\eta>0$ such that
   $|\varphi(x,j)-\varphi(z,j)|<\epsilon$
   for all $x,z\in H$ with  $\|x-z\|_H<\eta$ and $j\in S$.
Thus we get
\begin{align}\label{hdd}
\begin{split}
 &\sum\limits_{j \in S}\int_{K} {\E\left| {\varphi\left( {u^ {\varepsilon_0} \left( {t,s,\xi,j} \right),r_{s,j}(t)} \right)
 - \varphi\left( {u^{\varepsilon _n } \left( {t,s,\xi,j} \right),r_{s,j}(t)} \right)} \right|}
 \mu ^{\varepsilon _n }_s \left( {d\xi,j} \right)\\
  &= \sum\limits_{j \in S}\int_{K}
  \left (
  \int_Y
  \left|
   \varphi\left( {u^ {\varepsilon_0} \left( {t,s,\xi,j} \right),r_{s,j}(t)} \right) -
  \varphi\left( {u^{\varepsilon _n } \left( {t,s,\xi,j} \right ) ,r_{s,j}(t) }
  \right )  \right |   P(d\omega)
  \right )
   \mu ^{\varepsilon _n }_s \left( {d\xi,j} \right) \\
  &\quad+
   \sum\limits_{j \in S}\int_{K}
  \left (
  \int_{Y^C }
  \left|
   \varphi\left( {u^ {\varepsilon_0} \left( {t,s,\xi,j} \right),r_{s,j}(t)} \right) -
  \varphi\left( {u^{\varepsilon _n } \left( {t,s,\xi,j} \right ),r_{s,j}(t)  }
  \right )  \right |   P(d\omega)
  \right )
   \mu ^{\varepsilon _n }_s \left( {d\xi,j} \right) \\
  &\le 2\sup_{(\xi,j)\in H\times S} |\varphi(\xi,j)|
   \mathop {\sup }\limits_{(\xi,j) \in {K\times S}} P\left( {\| {u^{\varepsilon _n }
  \left( {t,s,\xi,j} \right)-u^ {\varepsilon_0} \left( {t,s,\xi,j} \right)} \|_H \ge \eta } \right) + \epsilon,
    \end{split}
 \end{align}
 where $Y=\left\{\omega\in \Omega|
  \| u^{\varepsilon _n }
  \left( {t,s,\xi,j} \right)-u^ {\varepsilon_0} \left( {t,s,\xi,j} \right)  \|_H
  \ge \eta \right \} $.

 It follows from $(A_3)$  and  \eqref{hcc}-\eqref{hdd} that
\begin{align}\label{hee}
\begin{split}
& \mathop {\lim   }\limits_{n \to \infty}
 \left| {\sum\limits_{j \in S}\int_{H} {\E \varphi\left( {u^{\varepsilon _0 }  \left( {t,s,\xi,j}
 \right),r_{s,j}(t)} \right)} \mu^{\varepsilon _n }_s \left( {d\xi,j} \right) -
 \sum\limits_{j \in S}\int_{H} {\varphi\left( \xi,j \right)\mu ^{\varepsilon _n }_t \left( {d\xi,j} \right)} } \right|\\
 &\quad \le   \epsilon  + 2\epsilon   \sup_{(\xi,j)\in {H\times S}} |\varphi(\xi,j)|.
  \end{split}
 \end{align}
 Since $\epsilon>0$ is arbitrary and $\mu^{\varepsilon_n}_t\rightarrow \mu_t$ weakly,
 by \eqref{hee} we
 obtain \eqref{haa_a1} immediately, which  shows
 that $\{\mu\}_{t\in \R}$ is an evolution system of   measures of the
process $y^{ \varepsilon_0}(t,s,\xi,i)$.
 \end{proof}

If for every $\varepsilon\in[0,1]$, $u^\varepsilon(t,s,\xi,j)$ have property $(A_2)$, we say
$u^\varepsilon(t,s,\xi,j)$ have property $(A_2)$.
Moreover,  we also assume

$(A_4)$  For any $\varepsilon\in [0,1]$, $s\in \R$,  $\xi\in H$ and $\eta>0$, there exists a
compact  $K=(\eta, \xi)\subset H$, independent of $\varepsilon$ and $s$, such that
for any  $j\in S$,
\[
P\left\{ {{u^\varepsilon\left( {t,s,\xi ,j} \right)}\in K,\quad t>s} \right\} <1- \eta.
\]
\begin{rem}
$(A_4)$  is stronger than $(A_1)$.

\end{rem}

Given  $\varepsilon  \in [0,1]$, for $A\subset \mathcal B(H)$ and $B\subset S$,
let $P^\varepsilon(t, s, \xi,  j, A \times B)$
denote the probability of event $\{y^\varepsilon(t,s,\xi,j)\in A \times B\}$
given initial condition $y^\varepsilon(s,s,\xi,j) = (\xi, j)$ at time $t=s$.
Denote by $(\mu^\varepsilon_t)_{t\in \R}$
the evolution system of measures of  $y^\varepsilon(t,s,\xi,j)$ obtained in the section above.
For each $\varepsilon  \in [0,1]$,  the definitions of operators $P^\varepsilon_{s,t}$ and $P^{\varepsilon,*}_{s,t}$
 with respect to $y^\varepsilon(t,s,\xi,j)$ are the same  as that of  $P_{s,t}$ and $P^{*}_{s,t}$ in Section 2.

\begin{thm}\label{cov1}
Suppose $(A_2)$-$(A_4)$  hold.
Then:

(i) For every $t\in \R$, the union
$\bigcup\limits_{\varepsilon \in [0,1]}
  \mu_t^\varepsilon $ is tight.

(ii)   If $\varepsilon_n \to \varepsilon_0 \in [0,1]$,
   then there exists a
subsequence $\varepsilon_{n_k}$ and a evolution system of measures
$\{\mu_t^{\varepsilon_0}\}_{t\in \R}$ of $y^{\varepsilon_0}(t,s,\xi,j)$
such that
$ \mu^{\varepsilon_{n_k}}_t \rightarrow \mu^{\varepsilon_0}_t$ weakly.
\end{thm}

\begin{proof}
$(i)$.
By  $(A_4)$ and the relationship $P_{s,t}^ {\varepsilon,*}  \delta _{\xi ,j} \left( {\Gamma  \times S} \right)
= P^\varepsilon\left( {t,s,\xi ,j,\left( {\Gamma  \times S} \right)} \right)$, for any $\Gamma\in \mathcal B(H)$,
it is easy to verify that
 the set $\bigcup\limits_{\varepsilon \in [0,1]}
  \mu^{\varepsilon }_t $ is tight.

$(ii)$.   By $(i)$ we know that $\{\mu^{\varepsilon_n}_t\}$, $t\in \R$, is tight,
and hence there exists a subsequence $\varepsilon_{n_k}$
and a probability measure $\mu^{*}_t$ such that
 $\mu^{\varepsilon_{n_k}}_t\rightarrow \mu^{*}_t $
weakly.
 It follows from Theorem \ref{Tmc} and
$(A_3)$ that $\{\mu_t^{\varepsilon_0}\}_{t\in \R}$
 is   a evolution system of measures
 of $y^{\varepsilon_0}(t,s,\xi,j)$.
This completes the proof.
\end{proof}

As an immediate consequence of Theorem \ref{cov1}, we have the following convergence result.

\begin{thm}\label{cov2}
Suppose $(A_2)$-$(A_4)$  hold.  Let $ \varepsilon_n,\varepsilon_0\in[0, 1]$
for all $n\in \N$ such that $\varepsilon_n\rightarrow \varepsilon_0$.
If $\{\mu_t^{\varepsilon_n}\}_{t\in \R}$ and $\{\mu^{\varepsilon_0}_t\}_{t\in \R}$  are the unique $\varpi$-periodic
evolution systems of measures of $\varpi$-periodic Markov processes  $y^{\varepsilon_n}(t,s,\xi,j)$
and   $y^{\varepsilon_0}(t,s,\xi,j)$, respectively, then  for each $t\in \R$,
$ \mu^{\varepsilon_{n }}_t \rightarrow \mu^{\varepsilon_0}$ weakly.
\end{thm}

\begin{proof}
 Note that in the present case, by Theorem \ref{Tper}, for every $\varepsilon\in [0, 1]$, $y^{\varepsilon}(t,s,\xi,j)$
has a unique $\varpi$-periodic evolution system of measures, which along with Theorem \ref{cov1} implies the desired result.
 \end{proof}

\section{Well-Posedness of stochastic lattice systems}
\setcounter{equation}{0}

In this section, we
prove  the existence  and uniqueness of solutions
to system \eqref{eu1}-\eqref{eu2}.
We first discuss the assumptions on the nonlinear
drift and diffusion terms in \eqref{eu1}.

Throughout this paper, we assume
the  sequences  $g(j)=(g_i(j))_{i\in \mathbb Z}$
and $h(j)=(h_{i,k}(j))_{i\in \mathbb Z,k\in \mathbb N}$, $j\in S$,
belong to  $l^2$:
\begin{equation}\label{gh}
\left\| g(j) \right\|^2  = \sum\limits_{i \in \mathbb Z} {\left| {g_i(j) } \right|^2 }  < \infty \quad \text{and} \quad
\left\| h(j) \right\|^2  = \sum\limits_{i \in \mathbb Z}
{\sum\limits_{k \in \mathbb N} {\left| {h_{i,k}(j) } \right|^2 } }  < \infty .
\end{equation}
where $\| \cdot \|$ is the norm of $l^2$.
The inner product of $l^2$ will be denoted by $(\cdot, \cdot)$
throughout this paper.

Assume that  for $j\in S$
$f_i :\R \times j\times \R \rightarrow \mathbb  R$,
$f_i =f_i (\cdot, j, \cdot)$,
is continuous in  $ \R \times \R$   and
  globally Lipschitz in $s
\in \R $
uniformly with respect to $i\in \mathbb Z$, $t\in \R$ and $j\in S$; more precisely,
 there exists a constant $L_f>0$ such that
$\text{for all}\,\, t, \, s_1 ,s_2
\in \R\,\, \text{and} \,\, i \in \mathbb Z$ and $j\in S,$
\begin{equation}\label{f1}
\left| {f_i \left( {t,j,s_1 } \right) - f_i \left( {t,j,s_2 } \right)} \right|
\le L_f\left| {s_1  - s^*_1 } \right|.
\end{equation}
Moreover,
  $f_i (t,j, s) $
 grows linearly in $s\in \R$:   for each $i\in \mathbb Z$, $t\in \R$ and $j\in S$, there exists
 $\alpha_i>0$     such that
\begin{equation}\label{f2}
\left| {f_i \left(t, j,s \right)} \right| \le \alpha _i
+ \beta _0   |s|,
\quad\forall \
 t, s \in \R\quad \text{and} \quad i \in \mathbb Z,
\end{equation}
where  $\left( {\alpha _i } \right)_{i \in \mathbb Z}
\in l^2$
and $\beta_0: \R \to \R$ is a positive constant.

For the   diffusion terms in \eqref{eu1}, we assume for $j\in S$
 $\sigma_{i,k}:\R \times j\times \R \rightarrow \mathbb R$,
 $\sigma_{i,k}=  \sigma_{i,k}(\cdot,j,\cdot)$,
 is continuous in $ \R \times \R$
 and  globally Lipschitz in $s\in  \R $
  uniformly with respect to $i\in \mathbb Z$, $t\in \R$ and $j\in S$;
 more precisely, for every $k\in \N$,
there exists a constant $L_k>0$ such that for all
 $  t,   s_1 ,s_2 \in \R,\,j\in S\,\,
\text{and} \,\, i \in \mathbb Z$
\begin{equation}\label{s1}
\left| {\sigma _{i,k} \left( {t,j,s_1 } \right) - \sigma _{i,k} \left( {t,j,s_2 } \right)} \right|
\le L_k\left| {s_1  - s_2 } \right|,
\end{equation}
where
 $\left( {L _k } \right)_{k \in \mathbb N}\in l^2$.
In addition,     we  assume
 $\sigma_{i,k}(t,j,s)$
   grows linearly in
   $ s\in \R $; that
is, for each   $t\in \R$, $j\in S$, $i \in \mathbb Z$  and $k\in \N$,
 there exists   $\delta_{i,k}>0$ and $\beta_k>0$
  such that
\begin{equation}\label{s2}
\left| {\sigma _{i,k} \left( {t,s,s^* } \right)} \right| \le \delta _{i,k}
+ \beta _k    \left| s \right|,\quad\forall s\in \R,
\end{equation}
where
 $\left( {\delta _{i,k} } \right)_{i \in \mathbb Z,k \in \N}
 \in l^2$  and
 $ (\beta_k(\cdot))_{k\in \N}\in l^2$
 is a positive continuous function.

When we will investigate the periodic evolution system of  measures of
system \eqref{eu1}-\eqref{eu2},   we assume that

$(P)$ All  given  time-dependent functions  are
$\varpi$-periodic in $t \in \R$ for some $\varpi>0$; that is,
  for  all $t\in \R $,
  $i\in \mathbb Z$ and $k\in \mathbb N$,
\[
\begin{array}{l}
  f_i \left( {t + \varpi, \cdot , \cdot } \right)
  = f_i \left( {t, \cdot , \cdot } \right), \quad
 \sigma _{i,k} \left( {t + \varpi, \cdot , \cdot } \right)
 = \sigma _{i,k} \left( {t, \cdot , \cdot } \right).
 \end{array}
\]

The following notation will be used throughout the paper:
\[
\alpha  = \left( {\alpha _i } \right)_{i \in \mathbb Z}
,\,\,L = \left( {L_k } \right)_{k \in \N} ,\,\,\beta
= \left( {\beta _k } \right)_{k \in \N} ,\,\,\delta  = \left( {\delta _{i,k} } \right)_{i \in \mathbb Z,k \in \N},
\]
\[
\left\| \alpha  \right\|^2  = \sum\limits_{i \in \mathbb Z} {\left| {\alpha _i } \right|^2 } ,\,\,\left\| L \right\|^2
 = \sum\limits_{k \in \N} {\left| {L_k } \right|^2 } ,\,\,\left\|
 \beta
   \right\|^2
 = \sum\limits_{k \in \N} {\left| {\beta _k  } \right|^2 } ,\,\,\left\| \delta  \right\|^2
  = \sum\limits_{i\in \mathbb Z} {\sum\limits_{k \in \N} {\left| {\delta _{i,k} } \right|^2 } } .
\]
and
\[
\lambda  = \mathop {\min }\limits_{j \in S} \lambda \left( j \right),\quad \left\| g \right\| =
\mathop {\max }\limits_{j \in S} \left\| {g\left( j \right)} \right\|,\quad \left\| h \right\| =
\mathop {\max }\limits_{j \in S} \left\| {h\left( j \right)} \right\|.
\]

In addition,  for
 $u=(u_i)_{i\in \mathbb Z}
 \in l^2$, we write
  $f(t,j,u)=(f_i(t,j,u_i)_{i\in \mathbb Z}$
  and
  $\sigma _k \left( t,j,u \right)
 = \left( {\sigma _{i,k} \left( {t,j,u_i }
  \right)} \right)_{i \in \mathbb Z}$.
It follows from \eqref{f1}-\eqref{f2} that
for all $t\in \R$, $j\in S$ and $u_1,u_2\in l^2,$
\begin{equation}\label{f3}
\left\| {f\left(t, j,u_1 \right) - f\left(t, j,u_2 \right)} \right\|^2 \le
 L_f^2 \left\| {u_1 - u_2} \right\|^2,
\end{equation}
and
\begin{equation}\label{f4}
\left\| {f\left(t, j,u_1 \right)} \right\|^2  \le 2\left\| \alpha  \right\|^2
+ 2\beta _0^2  \left\| u_1 \right\|^2.
\end{equation}

Similarly, by  \eqref{s1}-\eqref{s2},
we have   for all $t\in \R$, $j\in S$ and $u_1,u_2\in l^2$,
\begin{equation}\label{s4}
\sum\limits_{k \in \N} {\left\| {\sigma _k \left( t,j,u_1 \right) - \sigma _k \left(t,j,u_2\right)} \right\|} ^2
  \le \left\| L \right\|^2\left\| {u_1 - u_2} \right\|^2
\end{equation}
and
\begin{equation}\label{s3}
\sum\limits_{k \in \N} {\left\| {\sigma _k
 \left(t, u_1,v_1 \right)} \right\|} ^2  \le 2\left\| \delta  \right\|^2
 + 2\left\| \beta  \right\|^2 \left\| u_1 \right\|^2.
\end{equation}

For simplicity,  define   linear operators $
A,B,   :l^2  \to l^2$  by
$$
 \left( {Au} \right)_i  = -u_{i - 1}  + 2u_i - u_{{i  + 1}} ,
 \quad
 \left( {Bu} \right)_i  = u_{i + 1}  - u_i ,
 \quad i \in \mathbb Z,\,\,u=(u_i)_{i\in \mathbb Z} \in l^2 .
$$
Then,
system  \eqref{eu1}-\eqref{eu2} can be
put into the following  form  in $l^2$
for $s\in \R$:
\begin{align}\label{eu3}
\begin{split}
 du\left( t \right) &+ \nu Au\left( t \right)dt + \lambda(r(t)) u\left( t \right)dt
 = \left( {f\left( {t,r(t),u(t)} \right) + g(r(t))} \right)dt \\
  &\quad +\varepsilon \sum\limits_{k = 1}^\infty
  {\left( {h_k(r(t))  + \sigma _k \left( {t,r(t),u\left( t \right)} \right)} \right)}
   dW_k \left( t \right),\quad  t>s,
\end{split}
\end{align}
with initial condition
\begin{align}\label{eu4}
u\left( s \right) = \xi.
\end{align}

Similar to  \cite{Mao}
for   stochastic  equations with Markovian switching  in $\R^n$,
under conditions
\eqref{gh}-\eqref{s2},
we can show  that
  for any $\xi  \in L^2(\Omega,l^2)$,
  system \eqref{eu3}-\eqref{eu4} has a unique solution, which is
  written as   $u(t)$. To highlight the
initial values, we let $r_{s,j}(t)$ be the Markov chain starting from state $i\in S$ at $t =s $ and
denote by $u(t,s,\xi,j)$ the solution of Eq. \eqref{eu3}-\eqref{eu4}
with initial conditions $u (s,s,\xi,j) = \xi\in L^2(\Omega,l^2) $ and
$r(s) = j$. Moreover, for  any bounded  subset $B$ of $l^2$,
\[
\mathop {\sup }\limits_{\left( {\xi,j} \right) \in B \times S} \E\left[ {\mathop {\sup }\limits_{s \le \tau  \le t}
 \| {u\left( {\tau,s,\xi ,j} \right)} \|^2 } \right] < \infty\quad \quad \forall t \ge s,
\]
which together with Chebyshev's inequality implies that

$(\mathcal A_0)$  For
any $s\in \R$, $T>0$, bounded set $B\subset H$ and $\eta>0$, there exists a
constant $R=R(\eta,B,T)>0$, independent of $s$, such that
for any $\xi\in B$ and $j\in S$,
\[
P\left\{ {\|{u\left( {t,s,\xi ,j} \right)}\| \ge R},\, t\in [s,s+T] \right\} < \eta.
\]

In the sequence, we will follow the definition in Section 2 and 3. Sometimes, we need to replace $H$ with $l^2$.

\section{Uniform estimates}
\setcounter{equation}{0}

In this section, we derive uniform estimates of the solution of problem \eqref{eu3}-\eqref{eu4}
which are necessary for establishing the existence and stability of evolution system of  measures.
 We assume that
\begin{equation}\label{mu}
\lambda  > 1 + \beta _0^2  + 2\left\| \beta  \right\|^2
\end{equation}
and
\begin{equation}\label{uc1}
\left\| L \right\|^2  + 4\lambda ^{ - 1} L_f^2  < \frac{7}{4}\lambda .
\end{equation}

We first discuss  uniform estimates of solutions of problem \eqref{eu3}-\eqref{eu4}.

\begin{lem}\label{guji}
Suppose \eqref{gh}-\eqref{s2} and \eqref{mu} hold. Then for
any $s\in \R$,  $\xi\in l^2$ and $\eta>0$, there exists a  bounded
 subset $B=(\eta,\xi)$ of $l^2$, independent of $s$, such that
for any $0<\varepsilon\leq1$, $\xi\in l^2$, $j\in S$ and $ t>s$,
\[
P\left\{ {{u\left( {t,s,\xi ,j} \right)}\in B} \right\} >1-\eta.
\]
\end{lem}

\begin{proof}
By \eqref{eu3} and Ito's formula, we get for $t >s$
\begin{align}\label{asb2}
\begin{split}
& \E\left( {\left\| {u\left( t \right)} \right\|^2 } \right) + 2\nu \int_{s}^t
 {\E(\left\| {Bu\left( \tau \right)} \right\|^2)} d\tau+ 2 \int_{s}^t
 {\lambda(r(\tau))\E(\left\| {u\left( \tau \right)} \right\|^2)} d\tau
  \\
 & = \E\left( {\left\| {u \left( s\right)} \right\|^2 } \right) +
 2 {\int_{s}^t {\E\left( {u\left( \tau \right),f\left( {\tau,r(\tau),u\left( \tau\right)} \right)} \right)d\tau} }  \\
  &\quad+ 2 {\int_{s}^t {\E\left( {u\left( \tau \right),g\left( r(\tau) \right)} \right)} } ds
  +\varepsilon^2  {\sum\limits_{k = 1}^\infty  {\int_{s}^t {\E\left( \left\| {h_k \left( r(\tau) \right) +
   \sigma _k \left( {\tau,r(\tau),u\left( \tau \right)} \right)} \right\|^2\right)  d\tau} } }.
\end{split}
\end{align}
By \eqref{f4} and \eqref{s3} we get for $t >s$
\begin{align}\label{sb5}
\begin{split}
\E\left( {\left\| {u\left( t \right)} \right\|^2 } \right)  \le
\E\left( {\left\| {u\left( s \right)} \right\|^2 } \right)  - \varpi _1
 \int_{s}^t {E\left( {\left\| {u\left( \tau \right)}
 \right\|^2 } \right)} d\tau + \varpi _2 (t-s),
\end{split}
\end{align}
where
\begin{align}\label{o1}
\varpi _1  = 2\lambda-2-2\beta _0^2-4\left\| \beta  \right\|^2
\end{align}
and
$$
\varpi _2  = 2\left( {\left\| \alpha  \right\|^2  + {\left\| g \right\|} ^2
+ {\left\| h \right\|} ^2  + 2\left\| \delta  \right\|^2 } \right).
$$
 Then, we get  for $t\geq 0$,
\begin{align}\label{sb8}
\E\left( {\left\| {u\left( t \right)} \right\|^2 } \right) \le
\E\left( {\left\| {\xi } \right\|^2 } \right)
e^{-\varpi_1(t-s)} + \frac{\varpi _2}{\varpi _1},
\end{align}
which together with Chebyshev's inequality completes the proof.
\end{proof}

In the following, we prove that
 any two solutions of  \eqref{eu3}-\eqref{eu4} converge to each other.

\begin{lem}\label{Lsc}
Suppose \eqref{gh}-\eqref{s2} and \eqref{uc1} hold. Then  for any $s\in \R$, $\eta>0$ and  bounded
 subset $B$ of $H$, there exists a $T=T(\eta,B)$, independent of $s$, such that
for $\left( {\xi _1 ,\xi _2 ,j} \right) \in B \times B \times S,$
\[
P\left\{ {\| {u\left( {t,s,\xi _1 ,j} \right)- u\left( {t,s,\xi _2 ,j} \right)} \| <
 \epsilon } \right\} \ge 1 -\eta,\quad \forall t-s \ge {\rm{T}}.
\]
\end{lem}
\begin{proof}
For simplicity, we set $u_1(t)=u(t,s,\xi_1,j)$ and $u_2(t)=u(t,s,\xi_2,j)$
 for $t\geq s$.
By \eqref{eu3} we get, for $t\geq s$ and $\triangle t>0$,
\begin{align}\label{uc3}
 &\E\left( {\left\| {u_1\left( {t+\triangle t } \right) - u_2\left( {t+\triangle t } \right)} \right\|^2 } \right)
 -\E\left( {\left\| {u_1\left( {t } \right) - u_2\left( {t} \right)} \right\|^2 } \right) \nonumber\\
 &\le - 2\lambda
 \E \left (
  \int_t^{t+\triangle t} {\left\| {u_1\left( {s } \right) - u_2\left( {s } \right)} \right\|^2
   ds}  \right )  \nonumber\\
  &\quad+ 2\E\left( {\int_t^{t+\triangle t} {\left\| {u_1\left( {s } \right) - u_2\left( {s } \right)}
   \right\|\left\| {f\left( {s,r(s),u_1(s)} \right) - f\left( {s,r(s),u_2(s)} \right)} \right\|ds} } \right)  \nonumber\\
  &\quad+ \varepsilon^2\E\left( {\sum\limits_{k = 1}^\infty
   {\int_t^{t+\triangle t} {\left\| {\sigma _k \left( {s,r(s),u_1\left( {s}
  \right)} \right) - \sigma _k \left( {s,r(s),u_2\left( {s } \right)} \right)} \right\|^2 } } ds} \right).
 \end{align}
It follows from  \eqref{f3} and \eqref{s4} that for all $t\geq 0$,
\begin{align*}
 &D^+\E\left( {\left\| {u_1\left( {t } \right) - u_2\left( {t } \right)} \right\|^2 } \right)
 \le- \left( {\frac{7}{4}\lambda  - \left\| L \right\|^2-\frac{4}{\lambda }L_f^2 } \right) {\E\left( {\left\|
   {u_1\left( {t} \right) - u_2\left( {t} \right)} \right\|^2 } \right)}.
 \end{align*}
 which implies that
 \begin{align}\label{uc6}
\E\left( {\left\| {u\left( {t,s,\xi _1,j } \right) - u\left( {t,s,\xi _2,j } \right)} \right\|^2} \right)
 \le \E\left( {\left\| {\xi _1  - \xi _2 } \right\|^2 } \right)e^{ - \gamma (t-s)},
 \end{align}
 where $\gamma=  {\frac{7}{4}\lambda  - \left\| L \right\|^2-\frac{4}{\lambda }L_f^2 } $.
 Then  the desired inequality follows from \eqref{uc6} and Chebyshev's inequality  immediately.
  \end{proof}

Next, we derive uniform estimates on the tails of the solutions of \eqref{eu3}-\eqref{eu4}
which are crucial for establishing the tightness of the family of probability distributions of the solutions.

\begin{lem}\label{weibu}
Suppose \eqref{gh}-\eqref{s2} and \eqref{mu} hold. Then for every $s\in \R$,
 $\xi\in l^2$ and  $\eta>0,$ there exists a positive
integer $N=N(\xi,\epsilon)$, independent of  $s$, such that for all $0<\varepsilon\leq 1$, $n\ge N$ and $t\ge s$,
the solution $ u$ of \eqref{eu3}-\eqref{eu4}
satisfies,
$$\sum_{|i|\ge n} \E(|u_i(t,s,\xi,j)|^2)\le \eta.$$
\end{lem}

\begin{proof}
Let $\theta:\mathbb R\to\mathbb R$ be a smooth function such that $0\le\theta(s)\le 1$ for all $s\in\mathbb R$ and
\begin{equation}
\theta(s)=0,\quad \text{for}\ |s|\le 1,\ \text{and}\quad \theta(s)=1,\ \text{for}\ |s|\ge 2.
\end{equation}
Given $n\in\mathbb N$, denote by $\theta_n=(\theta(\frac{i}{n}))_{i\in\mathbb Z}$
and $\theta_nu=(\theta(\frac{i}{n})u_i)_{i\in\mathbb Z}$ for $u=(u_i)_{i\in\mathbb Z}$. By \eqref{eu3} we get
\begin{align}\label{te1}
\begin{split}
& d(\theta_n u(t))=-\theta_n\nu A u(t)dt-\theta_n\lambda(r(t))u(t)dt+\theta_nf(t, u(t))dt\\
&\quad+\theta_ng(t)dt+\sum_{k=1}^\infty(\theta_nh_k(t)+
\theta_n\sigma_k( t, u(t)))dW_k(t),\quad t>s.
\end{split}
\end{align}

By \eqref{te1}, Ito's formula and taking the expectation we obtain
for all $t >s$,
\begin{align}\label{te3}
\begin{split}
 &\E\left( {\left\| {\theta _n u\left( t \right)} \right\|^2 } \right) = \E\Big(\left\| {\theta _n
  u(s) } \right\|^2 \big)- 2 \int_{s}^t
   {\E\left( \theta _nA u\left( \tau \right),\theta _n u\left( \tau \right) \right)d\tau}\\
  &\quad- 2 \int_{s}^t
   {\E\left( \theta _n\lambda(r(\tau)) u\left( \tau \right),\theta _n u\left( \tau \right) \right)d\tau}\\
   &\quad+ 2\int_{s}^t {\E\left( {\theta _n u\left( \tau \right),\theta _n
   f\left( {\tau,r(\tau),u(\tau)} \right) + \theta _n g(r(\tau))} \right)d\tau}  \\
  &\quad+\varepsilon^2 \sum\limits_{k = 1}^\infty  {\int_{s}^t {\E\left( {\left\| {\theta _n h_k(r(\tau))
  + \theta _n \sigma _k \left( {\tau,r(\tau),u\left( \tau \right)} \right)} \right\|^2 } \right)} } d\tau.
  \end{split}
\end{align}
By the similar argument of the proof of Lemma 4.2 in \cite{W2019},
we have from \eqref{f4} and \eqref{s3} there exists $N=N(\xi,\eta)$, independent of  $s$,
 such that  for $0<\varepsilon\leq 1$, $t>s$  and $n\geq N$,
\begin{align}\label{te13}
 \E\left( {\left\| {\theta _n u\left( t \right)} \right\|^2 } \right)
  &\le c_1\epsilon
  -\varpi_1\int_s^t\E\left(
   {\left\| {\theta _n u\left( s \right)} \right\|^2 } \right)ds + c_2\epsilon,
\end{align}
where $\varpi_1$ is defined as  \eqref{o1} and $c_1, c_2>0$.
Then we get  for $t\geq s$  and $n\geq N$,
\begin{equation}\label{te15}
\E\left( {\left\| {\theta _n u\left( t \right)} \right\|^2 } \right) \le
c_1\epsilon e^{-\varpi_1(t-s)} + c_2\frac{\epsilon}{\varpi_1}\leq c\epsilon,
\end{equation}
where $c$ is independent of $t$, $s$ and $\epsilon$.
This completes the proof.

\end{proof}

\begin{lem}\label{Lpd}
Suppose \eqref{gh}-\eqref{s2}  hold. Then
for every  compact set $K\subseteq l^2$, $t\geq s$,  $\eta>0$
and $\varepsilon _1, \varepsilon _2\ge 0$,
\begin{equation}\label{u12}
\mathop {\lim }\limits_{\varepsilon _2  \to \varepsilon _1 } \mathop {\sup}\limits_{(\xi,j)  \in K\times S}
 P\left (
 {\| {u^{\varepsilon_2} \left( {t,s,\xi ,j} \right)-u^{\varepsilon_1}
 \left( {t,s,\xi,j } \right)} \|\ge \eta }
 \right) = 0.
\end{equation}
\end{lem}
\begin{proof}
The proof is similar as that of Lemma 6.2 in \cite{LWW2022}, so we omit it here.
\end{proof}

The transition operators $\{P_{\tau,t}\}_{s\le \tau\le t}$  have the
following properties.

\begin{lem}\label{markovp1}
 Assume that  \eqref{gh}-\eqref{s2}  hold. Then:

 $(i)$   $\left\{ {P_{\tau,t} } \right\}_{s\le \tau\le t} $ is Feller;

  $(ii)$ Under the additional condition $(P)$, the family $\left\{ {P_{r,t} } \right\}_{s \le r \le t} $
is $\varpi$-periodic; that is,  for all $s \le r \le t$,
\[
P\left( {r,t,\xi,j, (\cdot,\cdot) } \right) = P\left( {r+\varpi,t+\varpi,\xi,j, (\cdot,\cdot) } \right),
\quad\forall (\xi,j)  \in l^2\times S;
\]

$(iii)$  ${y \left( {t,s,\xi,j } \right)}$
 is a $(l^2\times S)$-valued nonhomogeneous  Markov process  for  evrty  $\xi  \in l^2$ and $j\in S$.

\end{lem}
\begin{proof}
The  properties $(i)$ and  $(iii)$
are standard  and the proof is  omitted.
 We only prove   $(ii)$.
By \eqref{eu3} we have for $t\geq \tau\geq s$ and $(\xi,j)  \in l^2\times S$,
\begin{align}\label{p1}
\begin{split}
 u\left( t, \tau, \xi,j \right) &+ \nu \int_\tau^t {Au\left( s, \tau, \xi,j \right)} ds
 +  \int_\tau^t {\lambda(r_{\tau,j}(s))u\left(  s, \tau, \xi,j \right)ds}\\
  &\quad
  = \xi  + \int_\tau^t \left( f ( s,r_{\tau,j}(s),u(s, \tau, \xi,j)  )
  + g (r_{\tau,j}(s))  \right) ds  \\
  &\quad+\varepsilon \sum\limits_{k = 1}^\infty
   {\int_\tau^t {\left( {h_k(r_{\tau,j}(s))
    +
  \sigma _k \left( {s,r_{\tau,j}(s),u\left( s,\tau, \xi,j \right)} \right)} \right)} }
   dW_k \left( s \right).
 \end{split}
\end{align}
We also have
\begin{align*}
 u\left( t+\varpi, \tau+\varpi, \xi,j  \right) &
 + \nu \int_{\tau+\varpi}^{t+\varpi} {Au\left( s,\tau+\varpi, \xi,j \right)} ds
 + \int_{\tau+\varpi}^{t+\varpi} {\lambda (r_{\tau+\varpi,j}(s))u\left( s ,\tau+\varpi, \xi,j\right)ds}\\
 &\quad
  = \xi  + \int_{\tau+\varpi}^{t+\varpi}
   {\left( {f\left( {s,r_{\tau+\varpi,j}(s),u(s,\tau+\varpi, \xi,j)} \right)
   + g(r_{\tau+\varpi,j}(s))} \right)ds}  \\
  &\quad+\varepsilon \sum\limits_{k = 1}^\infty
    {\int_{\tau+\varpi}^{t+\varpi} {\left( {h_k (r_{\tau+\varpi,j}(s)) +
  \sigma _k \left( {s,r_{\tau+\varpi,j}(s),u\left( s,\tau+\varpi, \xi,j \right)} \right)} \right)} }
   dW_k \left( s \right),
\end{align*}
which implies from $(P)$  that
\begin{align}\label{p2}
\begin{split}
 u & ( t+\varpi,\tau+\varpi, \xi,j  )+ \nu \int_{r}^{t} {Au\left( s+\varpi,\tau+\varpi, \xi,j \right)} ds \\
 & = \xi -\int_{\tau}^{t} { \lambda(r_{\tau+\varpi,j}(s+\varpi))u\left( s+\varpi,\tau+\varpi, \xi,j \right)ds}\\
 &\quad+ \int_{\tau}^{t} {\left( {f\left( {s+\varpi,r_{\tau+\varpi,j}(s+\varpi),u(s+\varpi,\tau+\varpi, \xi,j)} \right)
 + g(r_{\tau+\varpi,j}(s+\varpi))} \right)ds}  \\
  &\quad+\varepsilon \sum\limits_{k = 1}^\infty  {\int_{\tau}^{t}
   {\left( {h_k (r_{\tau+\varpi,j}(s+\varpi))  +
  \sigma _k \left( {s+\varpi,r_{\tau+\varpi,j}(s+\varpi),u\left( s+\varpi,\tau+\varpi, \xi,j \right)} \right)}
   \right)} }
   d\tilde W
_k \left( s \right)\\
    & = \xi  -\int_{\tau}^{t} { \lambda(r_{\tau,j}(s))u\left( s+\varpi,\tau+\varpi, \xi,j \right)ds}\\
 &\quad+ \int_{\tau}^{t}
    {\left( {f\left( {s,r_{\tau,j}(s),u(s+\varpi,\tau+\varpi, \xi,j)} \right)
    + g(r_{\tau,j}(s))} \right)ds} \\
  &\quad+\varepsilon \sum\limits_{k = 1}^\infty  {\int_{\tau}^{t}
  {\left( {h_k(r_{\tau,j}(s))  +
  \sigma _k \left( {s,r_{\tau,j}(s),u\left( s+\varpi,\tau+\varpi, \xi,j \right)} \right)} \right)} }
   d\tilde W_k \left( s \right),
   \end{split}
\end{align}
where $\tilde W_k \left( s \right)=W_k(s+\varpi)-W_k(\varpi)$, $k\in \N$,
are Brownian motions  as well.
Based on  \eqref{p1}-\eqref{p2},
one can show that
  $(u({t+\varpi},\tau+\varpi,\xi,j),r_{\tau+\varpi,j}(t+\varpi))$ and $(u(t,\tau,\xi,j),r_{\tau,j}(t))$
 have the same   distribution law.
Consequently,
for any $A\in \mathcal B\left(l^2 \right)$ and $j^*\in S$
\[
P\left\{ {y\left( {t+\varpi,\tau + \varpi,\xi,j } \right) \in (A,j^*)} \right\}
= P\left\{ {y\left( {t,\tau,\xi,j } \right) \in (A,j^*)} \right\},
\]
that is
\[
P\left( {t +\varpi,\tau+ \varpi,\xi,j,(A,j^*)} \right) = P\left( {t,\tau,\xi,j,(A,j^*)} \right),
\]
which completes the proof.
\end{proof}

\section{Main results}
\setcounter{equation}{0}

We are now in a position to prove the existence and stability of
evolution system of measures of \eqref{eu3}-\eqref{eu4}.
\begin{thm}\label{Tec}
Suppose \eqref{gh}-\eqref{s2} and \eqref{mu}-\eqref{uc1} hold.
Then \eqref{eu3}-\eqref{eu4} has an evolution system of measures $\{\mu_t\}_{t\in \R}$, which is
pullback and forward  asymptotic stability in distribution.

\end{thm}
\begin{proof}
The results of Lemma \ref{guji} and Lemma \ref{Lsc} collide  with the conditions $(A_1)$ and $(A_2)$
with $H$ replaced by $l^2$, respectively.
It follows from Lemma \ref{markovp1} that $y^\varepsilon(t,s,\xi,j)$ are  Markov processes
and the  transition  operator of $y^\varepsilon(t,s,\xi,j)$ is Feller.
By Theorem \ref{Lev}, Remark \ref{rpc}
and Theorem \ref{Lfc}, we complete the proof  immediately.

\end{proof}

\begin{thm}\label{Tpec}
Suppose $(P)$, \eqref{gh}-\eqref{s2} and \eqref{mu}-\eqref{uc1} hold.
Then \eqref{eu3}-\eqref{eu4} has a unique $\varpi$-periodic evolution system of
measures $\{\mu_t\}_{t\in \R}$, which is
pullback and forward  asymptotic stability in distribution.

\end{thm}
\begin{proof}
 Under conditions $(P)$ and  \eqref{gh}-\eqref{s2},
 it follows from Lemma \ref{markovp1} that $y^\varepsilon(t,s,\xi,j)$ are $\varpi$-periodic Markov processes.
By the similar argument as that in Theorem \ref{Tec}  and
Theorem \ref{Tper}, we complete the proof  immediately.

\end{proof}

Next,   we discuss the limiting behavior of
evolution system of  measures of \eqref{eu3}-\eqref{eu4}
as $\eps \to 0$
by applying   Theorem \ref{cov1}.
Note that all results in the previous sections are valid for $\varepsilon=0$ in which case the proof is actually
simpler.
For convenience,
 we  now write the solution
  of   \eqref{eu3}-\eqref{eu4} as  $u^\varepsilon(t,s,\xi,j)$
 with  $\varepsilon\in [0,1]$, $\xi\in l^2$ and $j\in S$
  and the evolution system of  measures of \eqref{eu3}-\eqref{eu4} obtained in \ref{Tec}
  as $\{\mu_t^\varepsilon\}_{t\in \R}$.

\begin{thm}\label{ctov}
Suppose \eqref{gh}-\eqref{s2} and \eqref{mu}-\eqref{uc1} hold.
Then:

(i) For every $t\in \R$, the union
$\bigcup\limits_{\varepsilon \in [0,1]}
  \mu_t^\varepsilon $ is tight.

(ii)   If $\varepsilon_n \to \varepsilon_0 \in [0,1]$,
   then there exists a
subsequence $\varepsilon_{n_k}$ and a evolution system of measures
$\{\mu_t^{\varepsilon_0}\}_{t\in \R}$ of $y^{\varepsilon_0}(t,s,\xi,j)$
such that
$ \mu^{\varepsilon_{n_k}}_t \rightarrow \mu^{\varepsilon_0}_t$ weakly.
\end{thm}
\begin{proof}
Since all uniform estimates
given in  Lemmas \ref{guji} and \ref{weibu}
are uniform with respect to $\varepsilon\in [0,1]$,
by  the similar  proof that of Lemma 4.2 in \cite{LWW2022},
it is easy to verify that
for any $s,t\in \R$,  $\xi\in l^2$ and $\eta>0$, there exists a
compact  $K=(\eta, \xi)\subset l^2$, independent of $s,t$ and $\varepsilon$, such that
for any  $j\in S$,
\[
P\left\{ {{u^\varepsilon\left( {t,s,\xi ,j} \right)}\in K,\quad s<t} \right\} >1- \eta.
\]
This means that the condition $(A_4)$ with $H$ replaced by $l^2$ is satisfied.
The results of   Lemma \ref{Lsc} and Lemma \ref{Lpd}
coincide  with the conditions  $(A_2)$  and $(A_3)$
with $H$ replaced by $l^2$, respectively.
By Theorem \ref{cov1}, we complete the proof.
\end{proof}

The following result about  the limiting behavior  of $\varpi$-periodic
 evolution system of measures of  $y^\varepsilon(t,s,\xi,j)$ follows
from Theorem  \ref{cov2} immediately.

\begin{thm}
Suppose $(P)$, \eqref{gh}-\eqref{s2} and \eqref{mu}-\eqref{uc1}   hold.
Let $ \varepsilon_n,\varepsilon_0\in[0, 1]$
for all $n\in \N$ such that $\varepsilon_n\rightarrow \varepsilon_0$.
If $\{\mu_t^{\varepsilon_n}\}_{t\in \R}$ and $\{\mu^{\varepsilon_0}_t\}_{t\in \R}$  are the unique $\varpi$-periodic
evolution systems of measures of $\varpi$-periodic of problem \eqref{eu3}-\eqref{eu4} with $\varepsilon$ replaced by $\varepsilon_n$
and   $\varepsilon_0$, respectively, then  for each $t\in \R$,
$ \mu^{\varepsilon_{n }}_t \rightarrow \mu^{\varepsilon_0}$ weakly.
\end{thm}

\end{document}